\documentclass[letterpaper,10pt,oneside,onecolumn,final]{amsart}
\usepackage{amsmath,amsfonts,amssymb}
\usepackage[bookmarksnumbered, colorlinks, plainpages]{hyperref}
\hypersetup{colorlinks=true,linkcolor=red, anchorcolor=green, citecolor=cyan, urlcolor=red, filecolor=magenta, pdftoolbar=true}

\newtheorem{theorem}{Theorem}[section]
\newtheorem{lemma}[theorem]{Lemma}
\newtheorem{proposition}[theorem]{Proposition}

\theoremstyle{definition}
\newtheorem{definition}[theorem]{Definition}

\theoremstyle{remark}
\newtheorem{remark}[theorem]{Remark}
\numberwithin{equation}{section}
\newcommand{\R}{\mathbb{R}}
\newcommand{\N}{\mathbb{N}}

\begin{document}
\setcounter{page}{1}

\title[Harmonically m-concave set-valued function]{Harmonically m-concave set-valued function}

\author[Gabriel Santana, Maira Valera-L\'opez, Nelson Merentes]{Gabriel Santana$^1$, Maira Valera-L\'opez$^{1}$, Nelson Merentes$^{1*}$}
\address{$^{1}$ Escuela de Matem\'atica,  Facultad de Ciencias, Universidad Central de Venezuela, Caracas 1010, Venezuela.}
\email{\textcolor[rgb]{0.00,0.00,0.84}{gaszsantana@gmail.com}}
\email{\textcolor[rgb]{0.00,0.00,0.84}{maira.valera@ciens.ucv.ve}}
\email{\textcolor[rgb]{0.00,0.00,0.84}{nmerucv@gmail.com}}


\date{Received: xxxxxx; Revised: yyyyyy; Accepted: zzzzzz.
\newline \indent $^{*}$ Corresponding author}

\begin{abstract}
	This research aimed to introduce the concept of harmonically m-concave  set-valued functions, which is obtained from the combination of two definitions: harmonically m-concave functions and set-valued functions. In this work some properties and characteristics are developed, as well as a Kuhn type theorem and Bernstein-Doetcsh type result for such functions.
\end{abstract} \maketitle

\section{Introduction}

With the intention to expand the study of harmonic set-valued functions, in this paper we extend the results obtained by G. Santana \textit{et. al.} in \cite{santana2020(2)}\\

\noindent In \cite{santana2020(1)}, harmonically m-convex set-valued functions are defined, in this paper we will introduce the counterpart of harmonically m-convex set-valued functions, as a result we will obtain a Kuhn theorem type and Bernstein-Doetsch type result.

\section{Preliminary}

As part of our research it is necessary to provide the reader with some preliminary definitions used throughout this investigation.

\begin{definition}(see \cite{Noor2016})
	Let $X$ a linear space. A nonempty subset $D$ of $X$ is said to be harmonic convex, if for all $x,y\in D$ and $t\in [0,1]$, we have:
	
	$$\frac{xy}{tx+(1-t)y}\in D.$$ 
	
	and for all $m\in (0,1]$ we say that $D$ is a harmonically $m$-convex set and satisfies
	 	$$\frac{mxy}{tmx+(1-t)y}\in D.$$
\end{definition}

\noindent In \cite{Iscan2014(1)}, \.{I}. \.{I}scan introduced a new variation of convexity called harmonically-convex function. A real function $f: D\rightarrow \R$ is harmonically-concave function if for all $x,y\in D$ and $t\in [0,1]$ then

\begin{equation}\label{HC2}
f\left(\frac{xy}{tx+(1-t)y}\right)\leq tf(y)+(1-t)f(x).
\end{equation}

\noindent If in the inequality (\ref{HC2}) we change ($\leq$) to ($\geq$) then we say that $f$ is {\bf harmonically concave}.\\

\noindent In the same way, M. A. Noor \textit{et. al.} (see \cite{Noor2016}) extends this definition to strongly convexity

\begin{definition}(see \cite{Noor2016})
	A  function $f:D\subset \R\backslash\{0\} \rightarrow \R$ is  said  to  be  strongly  harmonic convex function with modulus $c >0$, if
	\begin{equation*}
	f\left(\frac{xy}{tx+(1-t)y}\right)\leq tf(y)+(1-t)f(x)-ct(1-t)\left\|\frac{x-y}{xy}\right\|^{2}.
	\end{equation*}
	The function $f$ is said to be {\bf strongly harmonic concave function with modulus} $c >0$, if $-f$ is strongly harmonic convex function, that is
	\begin{equation}\label{hconcave}
	f\left(\frac{xy}{tx+(1-t)y}\right)+ct(1-t)\left\|\frac{x-y}{xy}\right\|^{2}\geq tf(y)+(1-t)f(x).
	\end{equation}
	 We say that $f$ is strongly harmonic mid concave with modulus $c$ if (\ref{hconcave}) is assumed only for $t=\frac{1}{2}$, that is
	 \begin{equation}
	 f\left(\frac{2xy}{x+y}\right)+\frac{c}{4}\left\|\frac{x-y}{xy}\right\|^{2}\geq \frac{f(y)+f(x)}{2}.
	 \end{equation}
\end{definition} 

In the same way \.{I}. \.{I}scan in \cite{Iscan2014(2)} extended the definition that we show before like that

\begin{definition}\label{iscan}(see \cite{Iscan2014(2)})
	The function $f : (0, b] \rightarrow \R,  b > 0$, is said to be harmonically $(\alpha, m)$-convex, where $\alpha \in [0, 1]$ and $m \in (0, 1]$, if 
	\begin{equation}\label{a-m-convex}
	f\left(\frac{mxy}{tmx+(1-t)y}\right)\leq t^{\alpha}f(y)+m(1-t^{\alpha})f(x).
	\end{equation}
	if in \ref{a-m-convex} we take $\alpha=1$, we say that $f$ is {\bf harmonically $m$-convex}.
\end{definition}

\begin{remark}
	In the following $\overline{B}$ denote the closed unit ball of $Y$ and $c>0$ a scalar. Also, we consider $X,Y$ linear spaces and we denote $c(Y),cc(Y), n(Y)$ as the convex subset, compact convex and not empty subsets of $Y$ respectively.
\end{remark}

\noindent The definition {\bf \ref*{iscan}} is used by G. Santana \textit{et. al.} in \cite{santana2020(1)} to introduced the definition of harmonically $m$-convex set-valued functions, combinating the notion of harmonically $m$-convex and set-valued functions, as

\begin{definition}(see \cite{santana2020(1)})
		Let $X$ and $Y$ be linear spaces, $D$ a harmonically $m$-convex subset of $X$ and $F:D\subset X\rightarrow n(Y)$ a set-valued function. It said that $F$ is said harmonically $m$-convex function if for all $x,y\in D$, $t\in [0,1]$, and $m\in (0,1]$, we have:
		
		\begin{equation}\label{FCVAmC_3}
		tF(y)+m(1-t)F(x)\subseteq F\left(\frac{mxy}{tmx+(1-t)y}\right).
		\end{equation}
\end{definition}

\noindent For the purposes of this paper it is necessary to mention the results obtained in \cite{santana2020(2)}, which correspond to harmonically concave set-valued function definition, Kuhn type theorem and Bernstein-Doetcsh type result for convex set-valued functions respectively. 

\begin{definition}(see \cite{santana2020(2)})
	Let $ X, Y $ linear spaces, and $ D \subset X $ a harmonically convex set, then we say that $ F: D \rightarrow n(Y) $ is a {\bf harmonically concave} set-valued function if for each $x_ {1}, x_ {2} \in D $ and $ t \in [0,1] $ we have
	
	\begin{equation}\label{hc4}
	F\left(\frac{xy}{tx+(1-t)y}\right)\subset tF(y)+(1-t)F(x)
	\end{equation}
	
	\noindent If we take $t=1/2$ in (\ref{hc4}) we say that $F$ is {\bf harmonically midconcave} set-valued function, that is 	
	
	$$F\left(\frac{2xy}{x+y}\right)\subset \frac{F(y)+F(x)}{2}$$
\end{definition}
 
\noindent In \cite{santana2020(2)} they define the harmonically concave set-valued functions and strongly harmonically concave set-valued functions with modulus $c$, and we obtain new results for this kind of functions as arithmetics operation, Kuhn type theorem and Bernstein-Doetcsh type result.     

\begin{definition}(see \cite{santana2020(2)})
	Let $t\in (0,1)$. A set-valued function $F: D\subset X\rightarrow n(Y)$ is strongly t-concave modulus $c$ if fol all $x,y\in D$ we have
	
	\begin{equation}\label{Concava}
		F\left(\frac{xy}{tx+(1-t)y}\right)+ct(1-t)\left\|\frac{x-y}{xy}\right\|^{2}\overline{B}\subset tF(y)+(1-t)F(x)
	\end{equation}
	
	\noindent If $t\in [0,1]$, then $F$ is strongly harmonically concave modulus $c$ and satisfies (\ref{Concava}).
\end{definition}

\begin{definition}(see \cite{santana2020(2)})
	Let $F$ a set-valued function. If $F$ satisfies (\ref{Concava}) for $t=1/2$, we say that $F$ is strongly harmonically midconcave modulus $c$ 
	
	$$F\left(\frac{2xy}{x+y}\right)+\frac{c}{4}\left\|\frac{x-y}{xy}\right\|^{2}\overline{B}\subset \frac{F(y)+F(x)}{2}$$
\end{definition}

The following theorem represent a Kuhn type result

\begin{theorem}\label{kuhn}(see \cite{santana2020(2)})
	Let $D$ a harmonically convex subset of $X$ and $t\in (0,1)$ a fixed point. If the set-valued function $F:D\rightarrow cc(Y)$ is strongly harmonically t-concave modulus $c$, then $F$ is strongly harmonically midconcave modulus $c$.
\end{theorem}

\noindent and the following is corresponding to Bernstein-Doetsch type result

\begin{lemma}\label{diadical}(see \cite{santana2020(2)})
	If $F:D\rightarrow cc(Y)$ is a strongly harmonically midconcave modulus $c$, then for all $x,y\in D$ and $k,n\in \N$ such that $k<2{n}$ we have
	
	$$F\left(\frac{xy}{\frac{k}{2^{n}}x+\left(1-\frac{k}{2^{n}}\right)y}\right)+c\frac{k}{2^{n}}\left(1-\frac{k}{2^{n}}\right)\left\|\frac{x-y}{xy}\right\|^{2}\overline{B}$$
	$$\subset \frac{k}{2^{n}}F(y)+\left(1-\frac{k}{2^{n}}\right)F(x)$$
\end{lemma}

\begin{theorem}\label{bernstein}(see \cite{santana2020(2)})
	Let $D$ be a harmonically convex subset. If $ F: D \rightarrow bcl(Y) $ is a strongly harmonically midconcave set-valued function modulus $c$ and semicontinuous superiorly over $D$, then $F$ is strongly harmonically concave modulus $c$.
\end{theorem}

\noindent So, based in this idea and results of \cite{santana2020(1)} we introduce the following definitions

\begin{definition}
	Let $ X, Y $ linear spaces, and $ D \subset X $ a harmonically $m$-convex set, then we say that $ F: D \rightarrow n(Y) $ is a {\bf harmonically $m$-concave} set-valued function if for each $x, y \in D $, $ t \in [0,1] $ and $m\in (0,1]$ we have
	
	\begin{equation}\label{hc}
	F\left(\frac{mxy}{tmx+(1-t)y}\right)\subset tF(y)+m(1-t)F(x)
	\end{equation}
	
	\noindent If we take $t=1/2$ in (\ref{hc}) we say that $F$ is {\bf harmonically $m$-midconcave} set-valued function, that is 	
	
	$$F\left(\frac{2mxy}{mx+y}\right)\subset \frac{F(y)+mF(x)}{2}$$
\end{definition}
\noindent In the next we define the strongly modulus $c$ definition of this kind of functions, as
\begin{definition}
	Let $X,Y$ linear space, $t\in (0,1)$, $m\in (0,1]$ and $D\subset X$ a harmonic $m$-convex subset. A set-valued function $F: D\subset X\rightarrow n(Y)$ is strongly $m-t$-concave modulus $c$ if fol all $x,y\in D$ we have
	
	\begin{equation}\label{conca}
	F\left(\frac{mxy}{tmx+(1-t)y}\right)+cmt(1-t)\left\|\frac{x-y}{xy}\right\|^{2}\overline{B}\subset tF(y)+m(1-t)F(x)
	\end{equation}
	
	\noindent If $t\in [0,1]$, then $F$ is strongly harmonically $m$-concave modulus $c$ and satisfies (\ref{conca}).
\end{definition}

\begin{definition}
	Let $F$ a set-valued function. If $F$ satisfies (\ref{conca}) for $t=1/2$, we say that $F$ is strongly harmonically $m$-midconcave modulus $c$ 
	
	$$F\left(\frac{2mxy}{mx+y}\right)+\frac{cm}{4}\left\|\frac{x-y}{xy}\right\|^{2}\overline{B}\subset \frac{F(y)+mF(x)}{2}$$
\end{definition}



\section{Main Results}

Now, let's see that the harmonically $m$-concave set-valued function satifies the sum, the product by scalar, the product with two kind of this functions, the union and the vectorial product.

\begin{proposition}
	Let $X,Y$ linear spaces in $\R$ and $F,G:X\rightarrow n(Y)$ are harmonically $m$-concave set-valued functions. Then :
	\begin{enumerate}
		\item $F+G$, is a harmonically $m$-concave set-valued function.
		\item $\lambda F$ con $\lambda\in \R$, harmonically $m$-concave set-valued function.
		\item If for each pair of points $x,y\in X$ we have that $F(x)\cap F(y)\neq \emptyset$ or $G(x)\cap G(y)\neq \emptyset$, then $(F\cdot G(x))$ is a harmonically $m$-concave set-valued function.
		\item $F\cup G$, is a harmonically $m$-concave set-valued function.
		\item $F \times G$, is a harmonically $m$-concave set-valued function. 
    \end{enumerate}
	
\end{proposition}

\begin{proof}{\bf (1)}
Let $x,y\in X$, $t\in [0,1]$ and $m\in (0,1]$. $F$ and $G$ harmonically $m$-concave set-valued functions, then

  $$F\left(\frac{mxy}{tmx+(1-t)y}\right)\subset tF(y)+m(1-t)F(x)$$
and
$$G\left(\frac{mxy}{tmx+(1-t)y}\right)\subset tG(y)+m(1-t)G(x)$$
Thus,
\begin{eqnarray*}
	(F+G)\left(\frac{mxy}{tmx+(1-t)y}\right) &=& F\left(\frac{mxy}{tmx+(1-t)y}\right)+ G\left(\frac{mxy}{tmx+(1-t)y}\right) \\
	&\subset& [tF(y)+m(1-t)F(x)]+[tG(y)+m(1-t)G(x)]  \\
	&=& t(F+G)(y)+m(1-t)(F+G)(x)
\end{eqnarray*}
\end{proof}

\begin{proof}{\bf (2)}
\noindent Let $\lambda\in \R$ then

\begin{eqnarray*}
	(\lambda F)\left(\frac{mxy}{tmx+(1-t)y}\right) &\subset& \lambda(tF(y)+m(1-t)F(x))  \\
	&=& t(\lambda F)(y)+m(1-t)(\lambda F)(x).
\end{eqnarray*}
\end{proof}

\begin{proof}{\bf (3)}

\noindent Before to prove that $F\cdot G$ is a harmonically concave set-valued function, lets see that if $x,y\in X$, then
$$F(y)G(x)+F(x)G(y)\subseteq F(y)G(y)+F(x)G(x)$$
Suppose without loss of generality that
$$F(x)\cap F(y)\neq \emptyset$$
Then
$$ \{0\}\subseteq [F(y)-F(x)]$$
where $\{0\}$ is the neutral element of $Y$ respect to sum. Thus
$$\{0\}\subseteq [F(y)-F(x)][G(y)-G(x)] $$

Then, as in $Y$ have the distributive property respect to the addition, the previous expression is equivalent to

$$\{0\} \subseteq  F(y)G(y)-F(y)G(x)-F(x)G(y)+F(x)G(x)$$

so, adding $F(y)G(x)$ y $F(x)G(y)$ remains
$$F(y)G(x)+F(x)G(y)\subseteq F(y)G(y)+F(x)G(x)$$

Now, $F\cdot G$ is a harmonically concave set-valued function, using (\ref{hc}) we will prove that. Indeed, for $t\in [0,1]$ y $x,y\in X$ we have the following

\begin{eqnarray*}
	(F\cdot G)\left(\frac{mxy}{tmx+(1-t)y}\right) &=& F\left(\frac{mxy}{tmx+(1-t)y}\right)\cdot G\left(\frac{mxy}{tmx+(1-t)y}\right) \\
	&=& [tF(y)+m(1-t)F(x)]\cdot[tG(y)+m(1-t)G(x)] \\
	&=& t^{2}F(y)G(y)+mt(1-t)[F(y)G(x)+F(x)G(y)]+(m(1-t))^{2}F(x)G(x) \\
	&\subset& t^{2}F(y)G(y)+mt(1-t)[F(y)G(y)+F(x)G(x)]+(m(1-t))^{2}F(x)G(x)  \\
	&=& [tF(y)G(y)+m(1-t)F(x)G(x)][t+m(1-t)] \\
	&\subset& t(F\cdot G)(y)+m(1-t)(F\cdot G)(x)
\end{eqnarray*}

\end{proof}

\begin{proof}{\bf (4)} 
	\begin{eqnarray*}
		(F\cup G)\left(\frac{mxy}{tmx+(1-t)y}\right) &=& F\left(\frac{mxy}{tmx+(1-t)y}\right)\bigcup G\left(\frac{mxy}{tmx+(1-t)y}\right) \\
		&=& [tF(y)+m(1-t)F(x)]\cup[tG(y)+m(1-t)G(x)] \\
		&=& t (F\cup G)(y)+ m(1-t)(F\cup G)(x)
	\end{eqnarray*}
\end{proof}	

\begin{proof}{\bf (4)} 
	\begin{eqnarray*}
	(F\times G)\left(\frac{mxy}{tmx+(1-t)y}\right) &=& F\left(\frac{mxy}{tmx+(1-t)y}\right)\times G\left(\frac{mxy}{tmx+(1-t)y}\right) \\
	&=& [tF(y)+m(1-t)F(x)]\times[tG(y)+m(1-t)G(x)] \\
	&=& t (F\times G)(y)+ m(1-t)(F\times G)(x)
	\end{eqnarray*}
\end{proof}

The following proposition we will used for every result in this paper. This provides a relation betwin harmonically concave and harmonically $m$-concave set-valued functions, but before we need the following proposition

\begin{proposition}\label{tD}(see \cite{santana2020(1)})
	Let harmonic $m$-convex ($m\neq 1$) subset $D$ of $X$ is said to be \textit{starshaped} if, for all $x$ in $D$ and all $t$ in the interval $(0, 1]$, the point $tx$ also belongs to $D$. That is:
	$$tD\subseteq D.$$ 
\end{proposition}

The following lemma will be useful for the results of this paper, this is well known as the cancellation law of H. R{\aa}dstrm (see \cite{Radstrm1952}) demonstrated in 1952.

\begin{lemma}\label{raa}
	Let $A_{1},A_{2},C$ subsets of $X$ such that $A_{1}+C\subset A_{2}+C$. If $A_{2}$ is convex closed and $C$ it is bounded not empty, then $A_{1}\subset A_{2}$.
\end{lemma}

\begin{proposition}\label{t-to-m}
  Let $X,Y$ linear Spaces, $D\subset X$ harmonic $m$-convex \textit{starshaped} and $F: X\rightarrow n(Y)$ a strongly harmonically $t$-concave set-valued function modulus $c$, then $F$ is strongly harmonically $m-t$-concave set-valued function modulus $c$, that is
  \begin{eqnarray*}
    F\left(\frac{mxy}{tmx+(1-t)y}\right)+cmt(1-t)\left\|\frac{x-y}{xy}\right\|^2\bar{B}&\subset& F\left(\frac{xy}{tx+(1-t)y}\right)+ct(1-t)\left\|\frac{x-y}{xy}\right\|^2\bar{B}\\
    &\subset& tF(y)+(1-t)F(x)\\
    &\subset& tF(y)+m(1-t)F(x)
  \end{eqnarray*}  	
\noindent for all $x,y\in D$, $t\in (0,1)$, $c> 0$ and $m\in (0,1]$.
\end{proposition}

\begin{proof}
	Note that for all $x,y\in D$ $m\in (0,1]$ and $t\in (0,1)$ we have
	$$\frac{mxy}{tmx+(1-t)y}\leq \frac{xy}{tx+(1-t)y}$$
	
	So, by the proposition \ref{tD} and since $F$ is harmonically $t$-concave set-valued function we have 
	$$F\left(\frac{mxy}{tmx+(1-t)y}\right) \subset F\left(\frac{xy}{tx+(1-t)y}\right)$$
	
	Now, by the Lemma \ref{raa} we have that for all $c> 0$
	$$F\left(\frac{mxy}{tmx+(1-t)y}\right)+cmt(1-t)\left\|\frac{x-y}{xy}\right\|^2\bar{B} \subset F\left(\frac{xy}{tx+(1-t)y}\right)+ct(1-t)\left\|\frac{x-y}{xy}\right\|^2\bar{B}$$
	
	In another hand we know that $F$ is strongly $t$-concave modulus $c$, then
	$$F\left(\frac{xy}{tx+(1-t)y}\right)+ct(1-t)\left\|\frac{x-y}{xy}\right\|^2\bar{B}\subset tF(y)+(1-t)F(x)$$
	
	As $F$ is concave we have that 
	$$F(x)\subset mF(x), \forall x\in D$$
	
	Thus, we have that we wanted to proof
	\begin{eqnarray*}
		F\left(\frac{mxy}{tmx+(1-t)y}\right)+cmt(1-t)\left\|\frac{x-y}{xy}\right\|^2\bar{B}&\subset& F\left(\frac{xy}{tx+(1-t)y}\right)+ct(1-t)\left\|\frac{x-y}{xy}\right\|^2\bar{B}\\
		&\subset& tF(y)+(1-t)F(x)\\
		&\subset& tF(y)+m(1-t)F(x)
	\end{eqnarray*}
 
\end{proof}	

Thus, we obtain the following result that represent the Kuhn type theorem

\begin{theorem}(Kuhn type theorem)
	Let $D$ a harmonically $m$-convex \textit{starshape} subset of $X$ and $t\in (0,1)$ a fixed point. If the set-valued function $F:D\rightarrow cc(Y)$ is strongly harmonically $m-t$-concave modulus $c$, then $F$ is strongly harmonically $m$-midconcave modulus $c$.
\end{theorem}

\begin{proof}
	From Proposition \ref{t-to-m} we have
	\begin{eqnarray*}
		F\left(\frac{mxy}{tmx+(1-t)y}\right)+cmt(1-t)\left\|\frac{x-y}{xy}\right\|^2\bar{B}&\subset& F\left(\frac{xy}{tx+(1-t)y}\right)+ct(1-t)\left\|\frac{x-y}{xy}\right\|^2\bar{B}\\
		&\subset& tF(y)+(1-t)F(x)\\
		&\subset& tF(y)+m(1-t)F(x)
	\end{eqnarray*}
   Now, the equation before it's true for $t=1/2$ by the Theorem \ref{kuhn}, so we have the desire result
   $$F\left(\frac{2mxy}{mx+y}\right)+\frac{cm}{4}\left\|\frac{x-y}{xy}\right\|^2 \bar{B}\subset \frac{F(y)+mF(x)}{2}$$ 
\end{proof}

The following Lemma generalizes the definition of harmonic $m$-midconcavity shown above.

\begin{lemma}\label{Diadical}
	If $F:D\rightarrow cc(Y)$ is a strongly harmonically $m$-midconcave modulus $c$, then for all $x,y\in D$, $m\in (0,1]$ and $k,n\in \N$ such that $k<2{n}$ we have
	
	$$F\left(\frac{mxy}{\frac{k}{2^{n}}mx+\left(1-\frac{k}{2^{n}}\right)y}\right)+cm\frac{k}{2^{n}}\left(1-\frac{k}{2^{n}}\right)\left\|\frac{x-y}{xy}\right\|^{2}\overline{B}$$
	$$\subset \frac{k}{2^{n}}F(y)+m\left(1-\frac{k}{2^{n}}\right)F(x)$$
\end{lemma}

\begin{proof}
	By the Lemma \ref{diadical} we have that
	$$F\left(\frac{xy}{\frac{k}{2^{n}}x+\left(1-\frac{k}{2^{n}}\right)y}\right)+c\frac{k}{2^{n}}\left(1-\frac{k}{2^{n}}\right)\left\|\frac{x-y}{xy}\right\|^{2}\overline{B}$$
	$$\subset \frac{k}{2^{n}}F(y)+\left(1-\frac{k}{2^{n}}\right)F(x)$$
	Now, using the proposition \ref{t-to-m} we have the following
\begin{eqnarray*}
&&F\left(\frac{mxy}{\frac{k}{2^{n}}mx+\left(1-\frac{k}{2^{n}}\right)y}\right)+cm\frac{k}{2^{n}}\left(1-\frac{k}{2^{n}}\right)\left\|\frac{x-y}{xy}\right\|^{2}\overline{B}\\
&\subset& F\left(\frac{xy}{\frac{k}{2^{n}}x+\left(1-\frac{k}{2^{n}}\right)y}\right)+c\frac{k}{2^{n}}\left(1-\frac{k}{2^{n}}\right)\left\|\frac{x-y}{xy}\right\|^{2}\overline{B}\\
&\subset& \frac{k}{2^{n}}F(y)+\left(1-\frac{k}{2^{n}}\right)F(x)\\
&\subset& \frac{k}{2^{n}}F(y)+m\left(1-\frac{k}{2^{n}}\right)F(x)
\end{eqnarray*}
\end{proof}

The Lemma \ref{Diadical} provides us with a valuable condition for the dyadical numbers between $[0,1]$, and we get a new result. Firs, consider the follow Lemma.\\

\begin{lemma}\label{lema2}
	For each set $A \subset Y$, the set-valued function $F:\R \rightarrow n(Y)$ defined by $ F(t) = tA, \hspace{0.2cm} \forall t \in \R $, is continuous in $\R$
\end{lemma}

\begin{theorem}
	Let $D$ be a harmonically $m$-convex subset. If $ F: D \rightarrow bcl(Y) $ is a strongly harmonically $m$-midconcave set-valued function modulus $c$ and upper semicontinuous on $D$, then $F$ is strongly harmonically $m$-concave modulus $c$.
\end{theorem}

\begin{proof}
Let $x, y \in D$ and $t \in (0,1)$. By the density of the diadycal numbers in $ [0,1] $, we can consider a succession of numerical numbers $\{q_{n}\}\subset (0,1)$ such that $ q_{n} \rightarrow t $. \\
Fixed $\epsilon> 0 $, exist $ n_{1} \in \N $ such that if $n \geq n_{1}$ then $|q_{n}-t| <\epsilon $. Let $ A\subset Y$ a bounded subset, then by the Lemma \ref{lema2} it is necessary that the set-valued function $ F: s \in \R \rightarrow sA$ is continuous, so if for all $ F(\gamma) \subset A$ with $ \gamma = \frac{mxy} {tmx + (1-t)y} \in D $, we have the following

\begin{equation}\label{1}
	q_{n}F(y)\subset tF(y)+\epsilon\overline{B}
\end{equation}
\begin{equation}\label{2}
	m(1-q_{n})F(x)\subset m(1-t)F(x)+\epsilon\overline{B}
\end{equation}
and
\begin{equation}\label{3}
	cmt(1-t)\left\|\frac{x-y}{xy}\right\|^{2}\overline{B}\subset cmq_{n}(1-q_{n})\left\|\frac{x-y}{xy}\right\|^{2}\overline{B}+\epsilon \overline{B}
\end{equation}
For all $n\geq n_{1}$. Now, given the semicontinuity above $F$ at the point $\gamma=\frac{mxy}{tmx+(1-t)y} \in D$ we have

\begin{equation}\label{4}
	F\left(\frac{mxy}{tmx+(1-t) y}\right)\subset F\left(\frac{mxy}{q_{n}mx+(1-q_{n}) y}\right)+ \epsilon\overline{B}
\end{equation}
For all $n\geq n_{2}$. Then, using (\ref{1}), (\ref{2}), (\ref{3}), (\ref{4}) and the Lemma \ref{Diadical} we have that

\begin{eqnarray*}
	F\left(\frac{mxy}{tmx+(1-t) y}\right)&+& cmt(1-t)\left\|\frac{x-y}{xy}\right\|^{2}\overline{B}\\
	&\subset& F\left(\frac{mxy}{q_{n}mx+(1-q_{n}) y}\right)+ cmq_{n}(1-q_{n})\left\|\frac{x-y}{xy}\right\|^{2}\overline{B}+2\epsilon \overline{B} \\
	&\subset& q_{n}F(y)+ m(1-q_{n})F(x)+2\epsilon \overline{B} \\
	&\subset& tF(y)+\epsilon\overline{B}+ m(1-t)F(x)+\epsilon\overline{B}+ 2\epsilon \overline{B}\\
	&=& tF(y)+m(1-t)F(x)+ 4\epsilon \overline{B}
\end{eqnarray*}

For all $n\geq \max\{n_{1},n_{2}\}$. Since the previous inclusions are kept for every $ \epsilon> 0 $, we have then
\begin{eqnarray*}
   F\left(\frac{mxy}{tmx+(1-t) y}\right)&+& cmt(1-t)\left\|\frac{x-y}{xy}\right\|^{2}\overline{B}\\
   &\subset& \bigcap_{\epsilon>0} tF(y)+m(1-t)F(x)+ 4\epsilon \overline{B}\\
   &=& (\overline{tF(y)+m(1-t)F(x)})\\
   &=& tF(y)+m(1-t)F(x)
\end{eqnarray*}

Thus, $F$ is a strongly harmonically $m$-concave set-valued function modulus $c$.
 
\end{proof}

 \end{document}